\documentclass[12pt, reqno]{amsart}
\usepackage{amsmath, amsthm, amscd, amsfonts, amssymb, graphicx, color}
\usepackage[bookmarksnumbered, colorlinks, plainpages]{hyperref}
\hypersetup{colorlinks=true,linkcolor=red, anchorcolor=green, citecolor=cyan, urlcolor=red, filecolor=magenta, pdftoolbar=true}

\newtheorem{theorem}{Theorem}[section]
\newtheorem{lemma}[theorem]{Lemma}
\newtheorem{corollary}[theorem]{Corollary}
\numberwithin{equation}{section}

\newcommand{\C}{\mathbb{C}}
\newcommand{\N}{\mathbb{N}}
\newcommand{\R}{\mathbb{R}}
\newcommand{\cB}{\mathcal{B}}
\newcommand{\cM}{\mathcal{M}}
\newcommand{\fM}{\mathfrak{M}}
\newcommand{\eps}{\varepsilon}

\begin{document}

\setcounter{page}{1}

\title[Noncompactness of Fourier convolution operators]%
{Noncompactness of Fourier Convolution
Operators on Banach Function Spaces}

\author[C. A. Fernandes, A. Yu. Karlovich, \MakeLowercase{and} Yu. I. Karlovich]%
{Cl\'audio A. Fernandes,$^1$ Alexei Yu. Karlovich,$^2$ \MakeLowercase{and} Yuri I. Karlovich$^{3}$}

\address{$^{1}$
Centro de Matem\'atica e Aplica\c{c}\~oes,
Departamento de Matem\'atica,
Faculdade de Ci\^encias e Tecnologia,
Universidade Nova de Lisboa,
Quinta da Torre,
2829--516 Caparica,
Portugal.}
\email{\textcolor[rgb]{0.00,0.00,0.84}{caf@fct.unl.pt}}

\address{$^{2}$
Centro de Matem\'atica e Aplica\c{c}\~oes,
Departamento de Matem\'atica,
Faculdade de Ci\^encias e Tecnologia,
Universidade Nova de Lisboa,
Quinta da Torre,
2829--516 Caparica,
Portugal.}
\email{\textcolor[rgb]{0.00,0.00,0.84}{oyk@fct.unl.pt}}

\address{$^{3}$
Centro de Investigaci\'on en Ciencias,
Instituto de Investigaci\'on en Ciencias B\'asicas y Aplicadas,
Universidad Aut\'onoma del Estado de Morelos,
Av. Universidad 1001, Col. Chamilpa,
C.P. 62209 Cuernavaca, Morelos, 
M\'exico.}
\email{\textcolor[rgb]{0.00,0.00,0.84}{karlovich@uaem.mx}}




\begin{abstract}
Let $X(\R)$ be a separable Banach function space such that
the Hardy-Littlewood maximal operator $M$ is bounded on $X(\R)$ and on its
associate space $X'(\R)$. Suppose $a$ is a Fourier multiplier on
the space $X(\R)$. We show that the Fourier convolution operator
$W^0(a)$ with symbol $a$ is compact on the space $X(\R)$ if and only
if $a=0$. This result implies that nontrivial Fourier convolution
operators on Lebesgue spaces with Muckenhoupt weights are never compact.
\end{abstract} \maketitle
\section{\textbf{Introduction}}
The set of all Lebesgue measurable complex-valued functions on $\R$ is denoted
by $\fM(\R)$. Let $\fM^+(\R)$ be the subset of functions in $\fM(\R)$ whose
values lie  in $[0,\infty]$. For a measurable set $E\subset\R$, 
its Lebesgue measure and the characteristic function are denoted by $|E|$ and
$\chi_E$, respectively. Following \cite[Chap.~1, Definition~1.1]{BS88}, a mapping
$\rho:\fM^+(\R)\to [0,\infty]$ is called a Banach function norm if,
for all functions $f,g, f_n \ (n\in\N)$ in $\fM^+(\R)$, for all
constants $a\ge 0$, and for all measurable subsets $E$ of $\R$,
the following properties hold:
\begin{eqnarray*}
{\rm (A1)} & & \rho(f)=0  \Leftrightarrow  f=0\ \mbox{a.e.}, \
\rho(af)=a\rho(f), \
\rho(f+g) \le \rho(f)+\rho(g),\\
{\rm (A2)} & &0\le g \le f \ \mbox{a.e.} \ \Rightarrow \ \rho(g)
\le \rho(f)
\quad\mbox{(the lattice property)},
\\
{\rm (A3)} & &0\le f_n \uparrow f \ \mbox{a.e.} \ \Rightarrow \
       \rho(f_n) \uparrow \rho(f)\quad\mbox{(the Fatou property)},\\
{\rm (A4)} & & |E|<\infty \Rightarrow \rho(\chi_E) <\infty,\\
{\rm (A5)} & & |E|<\infty \Rightarrow \int_E f(x)\,dx \le C_E\rho(f),
\end{eqnarray*}
where $C_E \in (0,\infty)$ may depend on $E$ and $\rho$ but is
independent of $f$. When functions differing only on a set of measure zero
are identified, the set $X(\R)$ of functions $f\in\fM(\R)$
for which $\rho(|f|)<\infty$ is called a Banach function space. For each
$f\in X(\R)$, the norm of $f$ is defined by
\[
\left\|f\right\|_{X(\R)} :=\rho(|f|).
\]
With this norm and under natural linear space operations, the set $X(\R)$ 
becomes a Banach space (see \cite[Chap.~1, Theorems~1.4 and~1.6]{BS88}). 
If $\rho$ is a Banach function norm, its associate norm $\rho'$ is defined on
$\fM^+(\R)$ by
\[
\rho'(g):=\sup\left\{
\int_{\R} f(x)g(x)\,dx \ : \ f\in \fM^+(\R), \ \rho(f) \le 1
\right\}, \quad g\in \fM^+(\R).
\]
By \cite[Chap.~1, Theorem~2.2]{BS88}, $\rho'$ is itself 
a Banach function norm.
The Banach function space $X'(\R)$ determined by the Banach function norm
$\rho'$ is called the associate space (K\"othe dual) of $X(\R)$.
The associate space $X'(\R)$ is naturally identified with a subspace of the (Banach) dual
space $[X(\R)]^*$.

Let $F:L^2(\R)\to L^2(\R)$ denote the Fourier transform
\[
(Ff)(x):=\widehat{f}(x):=\int_\R f(t)e^{itx}\,dt,
\quad
x\in\R,
\]
and let $F^{-1}:L^2(\R)\to L^2(\R)$ be the inverse of $F$,
\[
(F^{-1}g)(t)=\frac{1}{2\pi}\int_\R g(x)e^{-itx}\,d x,
\quad
t\in\R.
\]
It is well known that the Fourier convolution operator $W^0(a):=F^{-1}aF$
is bounded on the space $L^2(\R)$ for every $a\in L^\infty(\R)$.
Let $X(\R)$ be a separable Banach function space. Then $L^2(\R)\cap X(\R)$
is dense in $X(\R)$ (see Lemma~\ref{le:density} below). A function 
$a\in L^\infty(\R)$ is called a Fourier multiplier on $X(\R)$ if the 
convolution operator $W^0(a):=F^{-1}aF$ maps $L^2(\R)\cap X(\R)$ into 
$X(\R)$ and extends to a bounded linear operator on $X(\R)$. The function 
$a$ is called the symbol of the Fourier convolution operator $W^0(a)$. 
The set $\cM_{X(\R)}$ of all Fourier multipliers on  $X(\R)$ is a unital 
normed algebra under pointwise operations and the norm
\[
\left\|a\right\|_{\cM_{X(\R)}}:=\left\|W^0(a)\right\|_{\cB(X(\R))},
\]
where $\cB(X(\R))$ denotes the Banach algebra of all bounded linear operators
on the space $X(\R)$.

It is well known that the multiplication operator $aI$ by a function
$a\in L^\infty(\R)$ is compact on the space $L^2(\R)$ if and only if $a=0$
a.e. on $\R$ (see, e.g., \cite[Corollary~2.3.2]{SM93} for the case
of arbitrary nonatomic measure spaces). Since $A\mapsto F^{-1}AF$
for $A\in\cB(L^2(\R))$ is a similarity transformation, we see that the
Fourier convolution operator $W^0(a)$ cannot be compact on $L^2(\R)$ unless
its symbol $a$ is trivial.

The aim of this paper is to show that a nontrivial Fourier convolution
operator $W^0(a)$ with symbol $a\in \cM_{X(\R)}$ is never compact on a
Banach function space $X(\R)$ if the space $X(\R)$ satisfies a natural
condition formulated in terms of the boundedness of the Hardy-Littlewood
maximal operator $M$ on the space $X(\R)$ and on its associate space $X'(\R)$.

Recall that the (non-centered) Hardy-Littlewood maximal function $Mf$ of a
function $f\in L_{\rm loc}^1(\R)$ is defined by
\[
(Mf)(x):=\sup_{Q\ni x}\frac{1}{|Q|}\int_Q|f(y)|\,dy,
\]
where the supremum is taken over all intervals $Q\subset\R$ of finite length
containing $x$. The Hardy-Littlewood maximal operator $M$ defined by the rule 
$f\mapsto Mf$ is a sublinear operator.
\begin{theorem}[Main result]\label{th:main}
Let $X(\R)$ be a separable Banach function space such that the
Hardy-Littlewood maximal operator $M$ is bounded on $X(\R)$ and on its
associate space $X'(\R)$. Suppose that $a\in\cM_{X(\R)}$. Then the Fourier
convolution operator $W^0(a)$ is compact on the space $X(\R)$ if and only
if $a=0$ almost everywhere.
\end{theorem}
As a consequence of Theorem~\ref{th:main}, we formulate a corollary for the
case of weighted Lebesgue spaces $L^p(\R,w)$. A measurable function  
$w:\R\to[0,\infty]$ is called a weight if the preimage  $w^{-1}(\{0,\infty\})$ 
of the set $\{0,\infty\}$ has measure zero. Let $1<p<\infty$.  The weighted 
Lebesgue space
\[
L^p(\R,w)=\{f\in\fM(\R):fw\in L^p(\R)\}
\]
is equipped with the norm
\[
\|f\|_{L^p(\R,w)}=\left(\int_\R|f(x)|^pw^p(x)\,dx\right)^{1/p}.
\]
By the well-known Muckenhoupt theorem (see, e.g., \cite[Theorem~2]{M72},
\cite[Theorem~I]{CF74} and \cite[Chap. IV, Theorem~2.8]{GR85}), 
the Hardy-Littlewood maximal operator
$M$ is bounded on the weighted Lebesgue space $L^p(\R,w)$ if and only
if the weight $w$ belongs to the Muckenhoupt class $A_p(\R)$, that is, if
$w\in L_{\rm loc}^p(\R)$, $w^{-1}\in L_{\rm loc}^{p'}(\R)$ and
\begin{equation}\label{eq:Ap-real-line}
\sup_Q
\left(\frac{1}{|Q|}\int_Q w^p(x)\,dx\right)^{1/p}
\left(\frac{1}{|Q|}\int_Q w^{-p'}(x)\,dx\right)^{1/p'}
<\infty,
\end{equation}
where $1/p+1/p'=1$ and the supremum is taken over all intervals $Q\subset\R$
of finite length $|Q|$. Since $w\in L_{\rm loc}^p(\R)$ and 
$w^{-1}\in L_{\rm loc}^{p'}(\R)$, the weighted Lebesgue space $L^p(\R,w)$ is 
a separable Banach function space and $L^{p'}(\R,w^{-1})$ is its associate 
space (see, e.g., \cite[Lemma~2.4]{KS14}). It follows immediately from 
\eqref{eq:Ap-real-line} that $w\in A_p(\R)$ if and only if  
$w^{-1}\in A_{p'}(\R)$. Thus $w\in A_p(\R)$ is also equivalent to the 
boundedness of the Hardy-Littlewood maximal operator $M$ on the space 
$L^{p'}(\R,w^{-1})$. Combining these observations with Theorem~\ref{th:main}, 
we arrive at the following new result.
\begin{corollary}
Let $1<p<\infty$ and $w\in A_p(\R)$. Suppose $a\in\cM_{L^p(\R,w)}$. Then the
Fourier convolution operator $W^0(a)$ is compact on the weighted Lebesgue
space $L^p(\R,w)$ if and only if $a=0$ almost everywhere.
\end{corollary}
This result can be extended to the setting of weighted variable Lebesgue
spaces $L^{p(\cdot)}(\R,w)$ because a generalization of Muckenhoupt's
theorem for these spaces is available (see \cite[Theorem~1.3]{CDH11}
and \cite[Theorem~1.5]{CFN12}).

The paper is organized as follows. In Section~\ref{sec:auxiliary}, we collect
essentially known auxiliary results about the density of nice functions
in separable Banach function spaces, the Stechkin type inequality for
Fourier multipliers on Banach functions spaces, and the continuous
embedding of the algebra $\cM_{X(\R)}$ into the space $L^\infty(\R)$.

In Section~\ref{sec:proofs}, we prove that the sequence of 
convolution operators $\{W^0(\chi_n)\}_{n\in\N}$ tends strongly to the zero 
operator on the space $X(\R)$ as $n\to\infty$ if their symbols $\chi_n$ are 
characteristic functions of intervals of the form 
$[t-\delta_n(t),t+\delta_n(t)]$ with a fixed $t\in\R$ and some sequence of 
positive numbers $\{\delta_n(t)\}_{n\in\N}$ such that $\delta_n(t)\to 0$ as 
$n\to\infty$. For a compact operator $W^0(a)$, this implies that the sequence
$\{W^0(\chi_n a)\}_{n\in\N}$ tends to the zero operator uniformly. This fact
combined with the continuous embedding of $\cM_{X(\R)}$ into $L^\infty(\R)$
shows that the $L^\infty$-norm of $a$ is ``locally infinitesimal" near
each point $t\in\R$. Finally, a compactness argument shows that the
$L^\infty$-norm of $a$ is ``globally infinitesimal", whence $a=0$ a.e.
on $\R$, which completes the proof of Theorem~\ref{th:main}.
\section{\textbf{Auxiliary results}}\label{sec:auxiliary}
\subsection{Separability, absolute continuity of the norm and density of
\boldmath{$C_0^\infty(\R)$} in Banach function spaces}
Following \cite[Chap.~1, Definition~3.1]{BS88},
a function $f$ in a Banach function space $X(\R)$ is said to have absolutely
continuous norm in $X(\R)$ if $\|f\chi_{E_n}\|_{X(\R)}\to 0$ for every
sequence $\{E_n\}_{n\in\N}$ of measurable sets on $\R$ satisfying
$\chi_{E_n}\to 0$ a.e. on $\R$ as $n\to\infty$. Notice that the sets $E_n$
are not required to have finite measure. If all functions $f\in X(\R)$ have
absolutely continuous norm in $X(\R)$, then the space $X(\R)$ itself is said 
to have absolutely continuous norm.

From \cite[Chap.~1, Corollary~5.6]{BS88}, one can extract the following.
\begin{lemma}\label{le:absolute-coninuity-versus-separability}
A Banach function space $X(\R)$ is separable if and only if it has absolutely
continuous norm.
\end{lemma}
As usual, let $C_0^\infty(\R)$ denote the set of all infinitely differentiable
compactly supported functions on $\R$.
\begin{lemma}\label{le:density}
If $X(\R)$ is a separable Banach function space, then the sets $C_0^\infty(\R)$
and $L^2(\R)\cap X(\R)$ are dense in $X(\R)$.
\end{lemma}
\begin{proof}
Since $X(\R)$ is separable, it has absolutely continuous norm in view of
Lemma~\ref{le:absolute-coninuity-versus-separability}. It is easy to see that
the space $Y(\R)=L^2(\R)\cap X(\R)$ equipped with the norm
\[
\|f\|_{Y(\R)}=\max\{\|f\|_{L^2(\R)},\|f\|_{X(\R)}\}
\]
has absolutely continuous norm, too. Then $C_0^\infty(\R)$ is dense in $Y(\R)$ 
and in $X(\R)$ due to \cite[Lemma~2.10(b)]{KS14}. Since 
$C_0^\infty(\R)\subset Y(\R)\subset X(\R)$, we conclude that $Y(\R)$ is dense 
in $X(\R)$.
\end{proof}
\subsection{Convolution operators with symbols in the algebra \boldmath{$V(\R)$}}
Suppose that $a:\R\to\C$ is a function of finite total variation $V(a)$ given 
by
\[
V(a):=\sup \sum_{k=1}^n |a(x_k)-a(x_{k-1})|,
\]
where the supremum is taken over all partitions of $\R$ of the form
\[
-\infty<x_0<x_1<\dots<x_n<+\infty
\]
with $n\in\N$. The set $V(\R)$ of all functions of finite total variation
on $\R$ with the norm
\[
\|a\|_V:=\|a\|_{L^\infty(\R)}+V(a)
\]
is a unital non-separable Banach algebra.
\begin{theorem}\label{th:Stechkin}
Let $X(\R)$ be a separable Banach function space such that the
Hardy-Littlewood maximal operator $M$ is bounded on $X(\R)$ and on its
associate space $X'(\R)$. If a function $a:\R\to\C$ has a finite total 
variation $V(a)$, then the convolution operator $W^0(a)$ is bounded on 
the space $X(\R)$ and
\begin{equation}\label{eq:Stechkin}
\|W^0(a)\|_{\cB(X(\R))}
\le
c_{X}\|a\|_V
\end{equation}
where $c_{X}$ is a positive constant depending only on $X(\R)$.
\end{theorem}
This result follows from \cite[Theorem~4.3]{K15}.

For Lebesgue spaces $L^p(\R)$, $1<p<\infty$, inequality~\eqref{eq:Stechkin} is
usually called Stechkin's inequality, and the constant $c_{L^p}$ is
calculated explicitly:
\begin{equation}\label{eq:constant-in-Stechkin}
c_{L^p}=\|S\|_{\cB(L^p(\R))}=\left\{\begin{array}{ccc}
\tan\left(\frac{\pi}{2p}\right) &\mbox{if}& 1<p\le 2,
\\[3mm]
\cot\left(\frac{\pi}{2p}\right) &\mbox{if}& 2\le p<\infty,
\end{array}\right.
\end{equation}
where $S$ is the Cauchy singular integral operator given by
\begin{equation}\label{eq:Cauchy-singular-integral-operator}
(Sf)(x):=\frac{1}{\pi i}\lim_{\eps\to 0}\int_{\R\setminus(x-\eps,x+\eps)}
\frac{f(t)}{t-x}\,dt.
\end{equation}
We refer to \cite[Theorem~2.11]{D79} for the proof of \eqref{eq:Stechkin}
in the case of Lebesgue spaces $L^p(\R)$ with $c_{L^p}=\|S\|_{\cB(L^p(\R))}$
and to \cite[Chap. 13, Theorem 1.3]{GK92} for the calculation of the norm
of $S$ given in the second equality in \eqref{eq:constant-in-Stechkin}.
\subsection{Continuous embedding of the algebra of multipliers
\boldmath{$\cM_{X(\R)}$} into \boldmath{$L^\infty(\R)$}}
One of the main ingredients of the proof of Theorem~\ref{th:main} is the
following.
\begin{theorem}\label{th:continuous-embedding}
Let $X(\R)$ be a separable Banach function space such that the
Hardy-Littlewood maximal operator $M$ is bounded on $X(\R)$ and on its
associate space $X'(\R)$. If $a\in\cM_{X(\R)}$, then
\begin{equation}\label{eq:continuous-embedding}
\|a\|_{L^\infty(\R)}\le\|a\|_{\cM_{X(\R)}}.
\end{equation}
\end{theorem}
\begin{proof}
By \cite[Theorem~3.8]{KS14}, the Cauchy singular integral operator $S$
given by \eqref{eq:Cauchy-singular-integral-operator} is bounded on the
space $X(\R)$. Then, in view of \cite[Theorem~3.9]{KS14},
\begin{equation}\label{eq:AX}
\sup_{-\infty<a<b<\infty}
\frac{1}{b-a}\|\chi_{(a,b)}\|_{X(\R)}\|\chi_{(a,b)}\|_{X'(\R)}<\infty.
\end{equation}
If \eqref{eq:AX} is fulfilled, then inequality \eqref{eq:continuous-embedding}
follows from \cite[inequality (1.2) and Corollary~4.2]{KS18}.
\end{proof}
\section{\textbf{Proof of the main result}}\label{sec:proofs}
\subsection{Strong convergence to zero of convolution operators whose
symbols are characteristic functions of segments shrinking to a point}
For every $t\in\R$ and every $\delta>0$, let $\chi_{t,\delta}$ be the
characteristic function of the segment $[t-\delta,t+\delta]$
and let $\chi_{t,\delta}^*$ be the characteristic function of
$\R\setminus[t-\delta,t+\delta]$.
We will denote by $\operatornamewithlimits{s-lim}\limits_{n\to\infty}A_n$ the 
strong limit of a sequence of operators $\{A_n\}_{n\in\N}\subset\cB(X(\R))$.
\begin{theorem}\label{th:strong-convergence-to-zero}
Let $X(\R)$ be a separable Banach function space. If the Hardy-Littlewood
maximal operator $M$ is bounded on the space $X(\R)$ and on its associate
space $X'(\R)$, then for every point $t\in\R$ and every sequence
$\{\delta_n(t)\}_{n\in\N}$ of positive numbers such that
\begin{equation}\label{eq:strong-convergence-to-zero-1}
\lim_{n\to\infty}\delta_n(t)=0,
\end{equation}
we have
\begin{equation}\label{eq:strong-convergence-to-zero-2}
\operatornamewithlimits{s-lim}_{n\to\infty} W^0(\chi_{t,\delta_n(t)})=0
\end{equation}
on the space $X(\R)$.
\end{theorem}
\begin{proof}
Fix a point $t\in\R$ and a sequence $\{\delta_n(t)\}_{n\in\N}$ of positive
numbers satisfying \eqref{eq:strong-convergence-to-zero-1}. Since
$\chi_{t,\delta_n(t)}\in V(\R)$ for every $n\in\N$, it follows from
Theorem~\ref{th:Stechkin} that the operator $W^0(\chi_{t,\delta_n(t)})$
is bounded on the space $X(\R)$ for every $n\in\N$. By Lemma~\ref{le:density},
the set $C_0^\infty(\R)$ is dense in $X(\R)$. In view of this observation
and a well-known fact from Functional Analysis (see, e.g.,
\cite[Lemma~1.4.1(ii)]{RSS11}), in order to prove equality
\eqref{eq:strong-convergence-to-zero-2}, it is sufficient to show that for
every $f\in C_0^\infty(\R)$,
\begin{equation}\label{eq:strong-convergence-to-zero-3}
\lim_{n\to\infty} \left\|W^0(\chi_{t,\delta_n(t)})f\right\|_{X(\R)}=0.
\end{equation}

Let $S(\R)$ denote the Schwartz space of rapidly decreasing functions. If
$f\in C_0^\infty(\R)\subset S(\R)$, then its Fourier transform $\psi:=Ff$
belongs to $S(\R)$ (see, e.g., \cite[Corollary~2.2.15]{G14}). Hence
\begin{equation}\label{eq:strong-convergence-to-zero-4}
K_0:=\max_{x\in\R}|\psi(x)|<\infty,
\quad
K_1:=\max_{x\in\R}|\psi'(x)|<\infty.
\end{equation}
If $x\ne 0$, then integrating by parts, we get for all $n\in\N$,
\begin{align}
2\pi(W^0(\chi_{t,\delta_n(t)})f)(x)
&=
\int_{t-\delta_n(t)}^{t+\delta_n(t)}e^{-ix\tau}\psi(\tau)\,d\tau
\nonumber\\
&=
\frac{e^{-ix\tau}\psi(\tau)}{-ix}
\Bigg|_{\tau=t-\delta_n(t)}^{\tau=t+\delta_n(t)}
+
\frac{1}{ix}\int_{t-\delta_n(t)}^{t+\delta_n(t)}e^{-ix\tau}\psi'(\tau)\,d\tau.
\label{eq:strong-convergence-to-zero-5}
\end{align}
It follows from \eqref{eq:strong-convergence-to-zero-4} and
\eqref{eq:strong-convergence-to-zero-5} that for $x\in\R$ and $m,n\in\N$,
\begin{align}
&
2\pi\left|(W^0(\chi_{t,\delta_n(t)})f)(x)\right|
\nonumber\\
&\quad
=
\left|
\int_{t-\delta_n(t)}^{t+\delta_n(t)}e^{-ix\tau}\psi(\tau)\,d\tau
\right|\chi_{0,m}(x)
+
\left|
\int_{t-\delta_n(t)}^{t+\delta_n(t)}e^{-ix\tau}\psi(\tau)\,d\tau
\right|\chi_{0,m}^*(x)
\nonumber\\
&\quad
\le
\left(\int_{t-\delta_n(t)}^{t+\delta_n(t)}|\psi(\tau)|\,d\tau\right)
\chi_{0,m}(x)
\nonumber\\
&\quad\quad
+
\left(\frac{|\psi(t+\delta_n(t))|+|\psi(t-\delta_n(t))|}{|x|}
+
\frac{1}{|x|}\int_{t-\delta_n(t)}^{t+\delta_n(t)}|\psi'(\tau)|\,d\tau\right)
\chi_{0,m}^*(x)
\nonumber\\
&\quad
\le
2K_0\delta_n(t)\chi_{0,m}(x)
+
\left(\frac{2K_0}{|x|}+\frac{2K_1\delta_n(t)}{|x|}\right)
\chi_{0,m}^*(x).
\label{eq:strong-convergence-to-zero-6}
\end{align}
By \cite[Example~2.1.4]{G14}, if $|x|>1$, then
\begin{equation}\label{eq:strong-convergence-to-zero-7}
(M\chi_{0,1})(x)=\frac{2}{|x|+1}\ge\frac{1}{|x|}.
\end{equation}
It follows from \eqref{eq:strong-convergence-to-zero-6}--%
\eqref{eq:strong-convergence-to-zero-7} and the lattice property of the norm
of $X(\R)$ that for all $m,n\in\N$,
\begin{align}
\left\|(W^0(\chi_{t,\delta_n(t)})f)(x)\right\|_{X(\R)}
\le &
\frac{K_0\delta_n(t)}{\pi}\left\|\chi_{0,m}\right\|_{X(\R)}
\nonumber\\
&+
\frac{K_0+K_1\delta_n(t)}{\pi}
\left\|\chi_{0,m}^*(M\chi_{0,1})\right\|_{X(\R)}.
\label{eq:strong-convergence-to-zero-8}
\end{align}
Since the Hardy-Littlewood maximal operator $M$ is bounded on $X(\R)$
and since $\chi_{0,1}\in X(\R)$, we see that $M\chi_{0,1}\in X(\R)$. By
Lemma~\ref{le:absolute-coninuity-versus-separability}, the function
$M\chi_{0,1}$ has absolutely continuous norm in $X(\R)$.

Fix $\eps>0$. It is clear that $\chi_{0,m}^* \to 0$  a.e.
on $\R$ as $m\to\infty$. It follows from the absolute continuity of
the norm of $M\chi_{0,1}\in X(\R)$ that there exists an $m\in\N$ such that
\begin{equation}\label{eq:strong-convergence-to-zero-9}
\frac{K_0+1}{\pi}
\left\|\chi_{0,m}^*(M\chi_{0,1})\right\|_{X(\R)}
<\frac{\eps}{2}.
\end{equation}
On the other hand, it follows from \eqref{eq:strong-convergence-to-zero-1}
that there exists an $N\in\N$ such that for all $n>N$,
\begin{equation}\label{eq:strong-convergence-to-zero-10}
K_1\delta_n(t)<1
\end{equation}
and
\begin{equation}\label{eq:strong-convergence-to-zero-11}
\frac{K_0\delta_n(t)}{\pi}
\left\|\chi_{0,m}\right\|_{X(\R)} <\frac{\eps}{2}.
\end{equation}
Combining \eqref{eq:strong-convergence-to-zero-8}--%
\eqref{eq:strong-convergence-to-zero-11}, we see that
\[
\left\|(W^0(\chi_{t,\delta_n(t)})f)(x)\right\|_{X(\R)}<\eps
\quad\mbox{for all}\quad n>N,
\]
which implies \eqref{eq:strong-convergence-to-zero-3} for all
$f\in C_0^\infty(\R)$, and this completes the proof of
\eqref{eq:strong-convergence-to-zero-2}.
\end{proof}
\subsection{Proof of Theorem~\ref{th:main}}
Suppose $a\in\cM_{X(\R)}$. If $a=0$, then $W^0(a)$ is the zero operator,
whence it is compact.

Now assume that $W^0(a)$ is compact on the space $X(\R)$.
Consider an arbitrary segment $[x_0,y_0]\subset\R$. By
Theorem~\ref{th:strong-convergence-to-zero}, for every point $t\in[x_0,y_0]$
and every sequence of positive numbers $\{\delta_n(t)\}_{n\in\N}$
such that $\delta_n(t)\to 0$ as $n\to\infty$, we get
equality \eqref{eq:strong-convergence-to-zero-2} on the space $X(\R)$.
Since the operator $W^0(a)$ is compact, it follows from
\eqref{eq:strong-convergence-to-zero-2}
and \cite[Lemma~1.4.7]{RSS11} that for every $t\in[x_0,y_0]$,
\[
\lim_{n\to\infty}\left\|W^0(\chi_{t,\delta_n(t)}a)\right\|_{\cB(X(\R))}
=
\lim_{n\to\infty}\left\|W^0(\chi_{t,\delta_n(t)})W^0(a)\right\|_{\cB(X(\R))}
=0.
\]
Fix an arbitrary $\eps>0$. Then for every $t\in[x_0,y_0]$ there exists
an $n(t)\in\N$ such that
\[
\left\|W^0(\chi_{t,\delta_{n(t)}(t)}a)\right\|_{\cB(X(\R))}<\eps.
\]
Since $[x_0,y_0]$ is a compact set, we can extract a finite subcovering
$\gamma_1,\dots,\gamma_m$ from the open covering
\[
\bigcup_{t\in[x_0,y_0]}(t-\delta_{n(t)}(t),t+\delta_{n(t)}(t))
\]
of $[x_0,y_0]$, where each interval $\gamma_j$ is of the form
\[
\gamma_j=(t_j-\delta_{n(t_j)}(t_j),t_j+\delta_{n(t_j)}(t_j))
\]
and $t_j\in[x_0,y_0]$ for all $j\in\{1,\dots,m\}$. Thus
\begin{equation}\label{eq:proof-1}
\left\|W^0(\chi_{\gamma_j}a)\right\|_{\cB(X(\R))}<\eps
\quad\mbox{for all}\quad j\in\{1,\dots,m\}.
\end{equation}
By Theorem~\ref{th:continuous-embedding}, for all $j\in\{1,\dots,m\}$,
\begin{equation}\label{eq:proof-2}
\|\chi_{\gamma_j}a\|_{L^\infty(\R)}\le\|\chi_{\gamma_j}a\|_{\cM_{X(\R)}}
=
\left\|W^0(\chi_{\gamma_j}a)\right\|_{\cB(X(\R))}.
\end{equation}
Since $[x_0,y_0]\subset\gamma_1\cup\dots\cup\gamma_m$, taking into account
\eqref{eq:proof-1} and \eqref{eq:proof-2}, we get
\[
\left\|\chi_{[x_0,y_0]}a\right\|_{L^\infty(\R)}
\le
\|\chi_{\gamma_1\cup\dots\cup\gamma_m}a\|_{L^\infty(\R)}
=
\max_{1\le j\le m}\|\chi_{\gamma_j}a\|_{L^\infty(\R)}<\eps.
\]
Passing to the limit as $\eps\to 0$ in the above inequality, we see that
$a(t)=0$ for almost all $t\in[x_0,y_0]$. Since $[x_0,y_0]\subset\R$ was
chosen arbitrarily, we conclude that $a=0$ almost everywhere on $\R$.
\qed
\\
\\
{\bf Acknowledgments.}
This work was partially supported by the Funda\c{c}\~ao para a Ci\^encia e a
Tecnologia (Portu\-guese Foundation for Science and Technology)
through the project
UID/MAT/00297/2013 (Centro de Matem\'atica e Aplica\c{c}\~oes).



\begin{thebibliography}{XX}
\bibitem{BS88}
C. Bennett and R. Sharpley,
\textit{Interpolation of Operators},
Academic Press, Boston, 1988.

\bibitem{CF74}
R. R. Coifman and C. Fefferman, 
\textit{Weighted norm inequalities for maximal functions and singular integrals}, 
Studia Math. \textbf{51} (1974), 241--249.

\bibitem{CDH11}
D. Cruz-Uribe, L. Diening, and P. H\"ast\"o, 
\textit{The maximal operator on weighted variable Lebesgue spaces},
Frac. Calc. Appl. Anal. \textbf{14} (2011), 361--374.

\bibitem{CFN12}
D. Cruz-Uribe, A. Fiorenza, and C. J.  Neugebauer,
\textit{Weighted norm inequalities for the maximal operator on variable Lebesgue spaces},
J. Math. Anal. Appl. \textbf{394} (2012), 744--760.

\bibitem{D79}
R. Duduchava, 
\textit{Integral Equations with Fixed Singularities},
Teubner Verlagsgesellschaft, Leipzig, 1979.

\bibitem{GR85}
J. Garc{\'\i}a-Cuerva and J. Rubio de Francia, 
\textit{Weighted Norm Inequalities and Related Topics},
North-Holland, Amsterdam, 1985.

\bibitem{GK92}
I. Gohberg and N. Krupnik, 
\textit{One-Dimensional Linear Singular Integral Equations.
Vol. II: General Theory and Applications},
Birkh\"auser, Basel, 1992.

\bibitem{G14}
L. Grafakos, 
\textit{Classical Fourier Analysis}, 3rd ed.,
Springer, New York, 2014.

\bibitem{K15}
A. Yu. Karlovich, 
\textit{Maximally modulated singular integral operators and their applications
to pseudodifferential operators on Banach function spaces},
Contemp. Math. \textbf{645} (2015), 165--178.

\bibitem{KS18}
A. Karlovich and E. Shargorodsky, 
\textit{When does the norm of a Fourier multiplier dominate its $L^\infty$ norm?}
Proc. London Math. Soc. \textbf{118} (2019), 901--941.

\bibitem{KS14}
A. Yu. Karlovich and I. M. Spitkovsky, 
\textit{The Cauchy singular integral operator on weighted variable Lebesgue spaces},
Oper. Theor. Adv. Appl. \textbf{236} (2014), 275--291.

\bibitem{M72}
B. Muckenhoupt, 
\textit{Weighted norm inequalities for the Hardy maximal function},
Trans. Amer. Math. Soc. \textbf{165} (1972), 207--226.

\bibitem{RSS11}
S. Roch, P. A. Santos, and B. Silbermann, 
\textit{Non-Commutative Gelfand Theories. A Tool-kit for Operator Theorists
and Numerical Analysts},
Springer, Berlin, 2011.

\bibitem{SM93}
R. K. Singh and J. S.  Manhas,  
\textit{Composition Operators on Function Spaces},
North-Holland, Amsterdam, 1993.
\end{thebibliography}
\end{document}